\documentclass[a4paper,12pt]{article}
\usepackage{array,amsfonts,amsmath,graphicx,mathrsfs,amssymb,amsthm,latexsym}
\usepackage[latin1]{inputenc}
\def \msp {\vspace{-1ex}}

\newtheorem{thm}{Theorem}[section]
\newtheorem{lemma}[thm]{Lemma}

\def \cB {{\cal B}}

\def \cH {{\cal H}}

\def \cP {{\cal P}}

\def \Z {\mathbb Z}

\begin{document}
\title{\vspace{-10ex} ~~ \\
Maximum uniformly resolvable decompositions of  $K_v$ and $K_v - I$ into 3-stars and 3-cycles }

\author {Selda K\"{u}\c{c}\"{u}k\c{c}\.{i}f\c{c}\.{i}
\thanks{Corresponding author, phone: $+90 212 338 1523$, fax: $+90 212 338 1559$}
\thanks{Research supported by Scientific and Technological Research Council of Turkey Grant Number: 110T692}\\
\small \msp Department of Mathematics \\
\small \msp Ko\c{c} University \\
\small \msp Istanbul\\
\small Turkey\\
{\small \vspace{1ex} \tt skucukcifci@ku.edu.tr} \\  Salvatore Milici
\thanks{Supported by MIUR and by C. N. R. (G. N. S. A. G. A.), Italy}\\
\small \msp Dipartimento di Matematica e Informatica \\
\small \msp Universit\`a di Catania \\
\small \msp Catania\\
\small Italia\\
{\small \vspace{1ex} \tt milici@dmi.unict.it} \\
 Zsolt Tuza \thanks{Supported in part
   by the Hungarian Scientific Research Fund,
  OTKA grant T-81493.} \\
\small \msp Alfr\'ed R\'enyi Institute of Mathematics \\
\small \msp Hungarian Academy of Sciences \\
\small \vspace{-.5ex} Budapest \\
\small \vspace{-.5ex} and \\
\small \msp Department of Computer Science and Systems Technology \\
\small \msp University of Pannonia \\
\small \vspace{-.5ex} Veszpr\'em \\
\small Hungary \\
{\small \vspace{-5ex} \tt tuza@dcs.uni-pannon.hu}  }

\date{ }
\maketitle

\begin{abstract}
Let $K_v$ denote the complete graph of order $v$ and $K_v - I$
denote $K_v$ minus a 1-factor. In this article we investigate
uniformly resolvable decompositions of $K_v$ and $K_v-I$ into $r$
classes containing only copies of $3$-stars and $s$ classes
containing only copies of $3$-cycles.
 We completely determine the spectrum in the case where the
 number of resolution classes of 3-stars is maximum.
\end{abstract}

\vbox{\small
\vspace{5 mm}
 }

\section{Introduction
 }\label{introduzione}

Given a collection of graphs $\cH$, an {\em $\cH$-decomposition\/} of
a graph $G$ is a decomposition of the edges of $G$ into isomorphic
copies of graphs from $\cH$, the copies of $H\in\cH$ in the
decomposition are called {\em blocks}. Such a decomposition is
called {\em resolvable\/} if it is possible to partition the blocks
into {\em classes\/} $\cP_i$ such that every point of $G$ appears
exactly once in some block of each $\cP_i$.

A resolvable $\cH$-decomposition of $G$ is sometimes referred to as
an {\em \hbox{$\cH$-factorization} of\/ $G$}, a class can be called
a {\em $\cH$-factor of\/ $G$}. The case where $\cH$ is a single edge
($K_2$) is known as a {\em $1$-factorization of $G$} and it is well
known to exist for $G=K_v$ if and only if $v$ is even. A single
class of a $1$-factorization, a pairing of all points, is also known
as a {\em $1$-factor\/} or a {\em perfect matching}.

In many cases we wish to impose further constraints on
the classes of an $\cH$-decomposition.
 For example, a class is called {\em uniform\/} if every
block of the class is isomorphic to the same graph from $\cH$. Of
particular note is the result of Rees \cite{R} which finds necessary
and sufficient conditions for the existence of uniform $\{K_2,
K_3\}$-decompositions of $K_v$. Uniformly resolvable decompositions
of $K_v$ have also been studied in \cite{DQS}, \cite{GM}, \cite{HR},
 \cite{M}, \cite{MT}, \cite{S1} and \cite{S2}.
 Moreover, recently Dinitz, Ling and Danziger
\cite{DLD} have solved the question of the existence of a
uniformly resolvable decomposition of $ K_{v}$ into $r$ classes of
$K_{2}$ and $s$ classes of $K_4$ {\em in the case in which the
 number\/ $s$ of\/ $K_4$-factors
is maximum}.

\subsection{Definitions and notation}

For any four vertices $a_1,a_2,a_3,a_4$,
 let the {\em $3$-star\/}, $K_{1,3}$,
 be the simple graph with the vertex set
$\{a_1,a_2,a_3,a_4\}$ and the edge set $\{\{a_1,a_2\},
\{a_1,a_3\}, \{a_1,a_4\}\}$. In what follows, we will denote
 it
by
$(a_1;a_2,a_3,a_4)$.

Let $K_{m(n)}$ denote the complete multipartite graph with $m$ parts
each of size $n$, that is, $K_{m(n)}$ has the vertex set $
\{\bigcup_{i=1}^m X^{i}\}$ with $|X^{i}|=n$ for $i=1,2,\ldots ,m$ and
$X^{i}\cap X^{j}=\emptyset$ for $i\neq j$, and the edge set
$\{\{u,v\}: u\in X^{i}, \, v\in X^{j}, \,
 1\le i<j\le m\}$.

Let $C_{m(n)}$ denote the graph with the vertex set $ \{\bigcup_{i=1}^m
X^{i}\}$ with $|X^{i}|=n$ for $i=1,2,\ldots ,m$ and $X^{i}\cap
X^{j}=\emptyset$ for $i\neq j$, and the edge set $\{\{u,v\}: u\in
X^{i}, \, v\in X^{j}, \, i-j\equiv 1\pmod{m} $ or $ j-i\equiv 1
\pmod{m}\}$. For constructions below we shall also need
 the particular case $|X^i|=12$. Then let
 $X^{i}=\{x^{i}_{h}: h=0,1,\ldots ,11\}$ and for each
$j\in\{0,1,2,3\}$ let $X^{i}_{j}=\{x^{i}_{3j}, x^{i}_{3j+1},
x^{i}_{3j+2}\}$, so that $X^{i}= \bigcup_{j=0}^3X^{i}_{j}$.
Define, for each $i=1,2,\ldots ,m$ and $r,s\in\{0,1,2,3\}$, the
following sets of $3$-stars:
\begin{eqnarray}
R^{i}_{r,s} & = &
\{X^{i}_{r};X^{i+1}_{s},X^{i+1}_{s+1},X^{i+1}_{s+2}\} \nonumber \\
 & = &
\left\{\{x^{i}_{3r};x^{i+1}_{3s},x^{i+1}_{3s+1},x^{i+1}_{3s+2}\},
 \{x^{i}_{3r+1};x^{i+1}_{3s+3},x^{i+1}_{3s+4},x^{i+1}_{3s+5}\},
 \{x^{i}_{3r+2};x^{i+1}_{3s+6}, x^{i+1}_{3s+7},x^{i+1}_{3s+8}\}\right\}
  \nonumber
\end{eqnarray}
where superscript addition is meant modulo 12.

A resolvable $\cH$-decomposition of $K_{u(g)}$ is known as a
{\em resolvable group divisible design\/} $\cH$-RGDD {\em of type\/} $g^u$, where
the parts of size $g$ are called the groups of the design. When $\cH
= K_n$ we will call it an $n$-RGDD.
We shall use the terms ``point'' and ``vertex'' as synonyms.

\subsection{Our results}

In this paper we study the existence of a uniformly resolvable
decomposition of $K_v$ and of $K_v-I$, having the following type:
\begin{quote}
 $r$ classes containing only copies of 3-stars
 and $s$ classes containing only copies of 3-cycles.
\end{quote}
We will use the notation $(K_{1,3},K_3)$-URD$(v;r,s)$ for such
 a uniformly resolvable decomposition of $K_{v}$ when $v$ is odd,
  and for that of $K_v-I$ when $v$ is even.
We will specify whether the system is
 a decomposition of $K_v$ or of $K_v-I$ only when
it is not clear in the context whether $v$ is odd or even.
 Further, we will use the
notation $(K_{1,3}, K_3)$-URGDD$(r,s)$ of type $g^{u}$  to denote a
uniformly resolvable decomposition of the complete multipartite
graph with $u$ parts each of size $g$ into $r$ classes containing
only copies of 3-stars and $s$ classes containing only copies of
$3$-cycles. As $r$ determines $s$ if $v$ is fixed,
 we will also use the simplified notation $K_{1,3}$-RGDD$(r)$ for
$(K_{1,3}, K_3)$-URGDD$(r,s)$ when $v$ is understood.

Determining the spectrum of triples $(v,r,s)$ which admit a
 $(K_{1,3},K_3)$-URD$(v;r,s)$ appears to be
  a rather hard problem in general.
Similarly to the work \cite{DLD}, here we concentrate on
 the extremal case in which the number of resolution
classes of 3-stars is maximum.
In particular, we will prove the following result in this paper:\\

\noindent \textbf{Main Theorem.} {\em For each\/  $v\equiv
0\pmod{12}$, there exists a\/ $(K_{1,3},
K_3)$-URD$(v;\frac{2(v-6)}{3},2)$ of\/ $K_v-I$. }

\section{Necessary conditions}

In this section we will give necessary conditions for the existence
of a uniformly resolvable decomposition of $K_v$ and $K_v-I$ into
$r$ classes of 3-stars and $s$ classes of
3-cycles.

\begin{lemma}
\label{lemmaP0} A $(K_{1,3},K_3)$-URD$(v;r,s)$, with\/ $r>0$
and\/ $s>0$, does not exist for any\/ $v\geq4$ of\/ $K_v$.
\end{lemma}

\begin{proof}
Assume that there exists a $(K_{1,3},K_3)$-URD$(v;r,s)$ $D$ of
$K_{v}$ with $r>0$ and $s>0$. By resolvability it follows that
$v\equiv 0\pmod{12}$, say $v=12u$. Counting the edges of $K_v$ that
appear in $D$ we obtain

$$\frac{3rv}{4}+\frac{3sv}{3}=\frac{v(v-1)}{2}$$
 and hence
\begin{equation}
  3r+4s=2(v-1). \end{equation}

The equality (1) implies that  $r\equiv2\pmod{4}$ and
$s\equiv1\pmod{3}$. Let $r=2+4t$ and $s=1+3h$ with $t,h\geq0$.
Denote by $B$ the set of the $r$ parallel classes of $3$-stars and
by $R$ the set of the $s$ parallel classes of $3$-cycles. Since the
classes of $R$ are regular of degree $2$, we have that every vertex
$x$ of $K_v$ is incident with $2s$ edges in $R$ and
$(12u-1)-(2+6h)=12u-6h-3$ edges in $B$. Assume that the vertex $x$
appears in $a$ classes
 with degree $3$ and in $b$ classes with degree 1
 in $B$. Since

$$a+b=2+4t \ \ \mbox{and}\ \
3a+b=12u-6h-3,$$ it follows that

$$2a=12u-6h-3-2-4t=2(6u-3h-1-2t)-3,$$ which is a contradiction,
since $2a$ cannot be odd.
\end{proof}

Given $v\equiv 0\pmod{12}$, define

$$J(v)=\{(4x,\frac{v-2}{2}-3x): x=1,\ldots ,\frac{v-6}{6}\}.$$

\begin{lemma}
\label{lemmaP1} If there exists a\/ $(K_{1,3}, K_3)$-URD$(v;r,s)$ of\/
$K_v-I$ with\/ $r>0$ and\/ $s>0$
 then\/ $v\equiv 0\pmod{12}$ and\/ $(r,s)\in J(v)$.
\end{lemma}
\begin{proof}
The condition $v\equiv 0\pmod{12}$ is trivial by the assumption
 that both $r$ and $s$ are positive. Let $D$ be a
$(K_{1,3}, K_3)$-URD($v;r,s)$ of  $K_{v}-I$. Counting the edges of
$K_v- I$ that appear in $D$ we obtain
$$\frac{3rv}{4}+\frac{3sv}{3}=\frac{v(v-2)}{2},$$
\\ and hence that
\begin{equation}
3r+4s=2(v-2).
 \end{equation}
This equality implies that $r\equiv0\pmod{4}$ and
$s\equiv2\pmod{3}$.
 Letting now $r=4x$, the value of $s$ is
   determined by $(2)$ as $s=\frac{v-2}{2}-3x$, where
   $3x\le \frac{v-2}{2}$ must hold and $x$ must be an integer.
Thus, $x\le \frac{v-6}{6}$ since $v$ is a multiple of~$6$.
 This
completes the proof.
\end{proof}

\section{Constructions and related structures}

In this section we will introduce some useful results and discuss
constructions we will use in proving the main result. For missing
terms or results that are not explicitly explained in the paper,
the reader is referred to \cite{CD} and its online updates.
For some results below, we also cite this handbook instead of the
 original papers.

A resolvable $K_3$-decomposition of $K_v$ is called a {\em Kirkman
Triple System} (KTS($v$)) and it is well known to exist if and
only if $v\equiv 3$ (mod $6$) \cite{CD}.
  Let $u>1$ be an integer.
 A $2$-RGDD of type $g^u$
exists if and only if $gu$ is even. Moreover, a $3$-RGDD of type
$g^u$ exists if and only if $g(u-1)$ is even and $gu\equiv 0$ (mod
$3$), except when $(g,u)\in\{(2,6), (2,3), (6,3)\}$ \cite{RS}.  In
particular, a $3$-RGDD of type $2^u$ is called a {\em Nearly Kirkman
Triple System} (NKTS($2u$)) and is known to exist whenever $u\equiv
0$ (mod $3$), $u > 6$ \cite{RS}.

We now recall the existence of
some $4$-RGDDs we will need in the proof.

\begin{lemma}
\label{lemma D0} \cite{CD} There exists a\/ $4$-RGDD of type
\begin{itemize}
\item $6^{t}$  for each\/ $t\equiv 0\pmod{2}$, $t > 4$, except when\/
$t\in\{6,54,68\}$;

\item $12^{t}$ for each\/  $t\geq 4$, except when\/
$t=27$.

\end{itemize}
\end{lemma}

We also need the following definitions. Let $(s_1, t_1)$ and $(s_2, t_2)$
be two pairs of non-negative integers. Define $(s_1, t_1) +(s_2,
t_2)=(s_1+s_2, t_1+t_2)$. If $X$ and $Y$ are two sets of pairs of
non-negative integers, then $X+Y$ denotes the set $\{(s_1, t_1)
+(s_2, t_2) : (s_1, t_1)\in X, (s_2, t_2) \in Y \}$. If $X$ is a set
of pairs of non-negative integers and $h$ is a positive integer,
then  $h * X$ denotes the set of all pairs of non-negative integers
which can be obtained by adding any $h$ elements of $X$ together
(repetitions of elements of $X$ are allowed).

\begin{thm}
\label{thmDD} Let\/ $v$, $g$, $t$, $u$ and\/ $x$  be non-negative
integers such that\/ $v=gtu$ and\/ $x\in\{3,4\}$. If there exists
\begin{itemize}
\item
[$(1)$] an\/ $x$-RGDD of type\/ $g^{u}$;

\item
[$(2)$] a $(K_{1,3},K_3)$-URGDD$(r_1,s_1)$ of\/
$K_{x(t)}$ with\/ $(r_1, s_1)\in J_1$;

\item
[$(3)$] a $(K_{1,3},K_3)$-URD$(gt;r_2,s_2)$  of\/ $K_{gt}- I_i$,\/
$i=1,2,\ldots ,u$, with\/ $(r_2, s_2)\in J_2$;
\end{itemize}
then there exists a\/ $(K_{1,3}, K_3)$-URD$(v;r,s)$ of\/ $K_{v}- I$ for
each\/ $(r,s)\in J_2+ h\ast J_1$, where\/ $h=\frac{g(u-1)}{x-1}$ is the
number of parallel classes of the\/ $x$-RGDD of type\/ $g^{u}$ and\/ $I=
\bigcup_{i=1}^u I_{i}$.

\end{thm}
\begin{proof}
Let $(X,\{G_1, \ldots, G_u\},\cB)$ be an $x$-RGDD of type $g^{u}$,
where the $G_i$, $i=1,2,\ldots ,u$, are the groups of size $g$, and
$x\in\{3,4\}$. Let $R_1,\ldots,R_{\frac{g(u-1)}{x-1}}$ be the
parallel classes of this $x$-RGDD. Give weight $t$ to all points of
this $x$-RGDD and place on each block of a given resolution class of
$\cB$ the same $(K_{1,3},K_3)$-URGDD$(r_1,s_1)$ with $(r_1, s_1)\in
J_1$. For each $i=1,\ldots,u$, place on $G_i\times\{1,\ldots,t\}$
the same $(K_{1,3},K_3)$-URD$(gt;r_2,s_2)$ of\/ $K_{gt}- I_i$,\/
$i=1,2,\ldots ,u$, with\/ $(r_2, s_2)\in J_2$. The result is a
$(K_{1,3}, K_3)$-URD$(v;r,s)$ of  $K_{v}-I$ for each $(r,s)\in J_2+
h\ast J_1$, where $h=\frac{g(u-1)}{x-1}$  is the number of parallel
classes of the $x$-RGDD of type $g^{u}$ and $I=\cup_{i=1}^u I_{i}$.
\end{proof}

\begin{thm}
\label{thmFFF} If there exists $(K_{1,3}, K_3)$-URGDD$(r,s)$
of\/ $K_{x(3)}$, then for each\/ $t\geq3$ there exists a $(K_{1,3}, K_3)$-URGDD$(r,s)$
of\/ $K_{(xt)(3)}$ into\/ $rt$ parallel classes
of\/ $3$-stars and\/ $rt$ parallel classes of\/ $3$-cycles.

\end{thm}
\begin{proof}
Let $K_1, \ldots, K_r $ be the parallel classes of $K_{x(3)}$
containing only copies of 3-stars and $C_1, \ldots, C_s $ be the
parallel classes of $K_{x(3)}$ containing only copies of $3$-cycles.
Give weight $t$ to all points of this $K_{x(3)}$ and for each block
$(a;b,c,d)$ of a given resolution class of $K_i$, $i=1,\ldots,r$,
construct  $t$ parallel classes of 3-stars on $\{a,b,c,d\}\times
\{1,\ldots,t\}$\,:
$$\{(a_i;b_{i+j-1},c_{i+j-1},d_{i+j-1})\}, \quad j=1,\ldots,t.$$

For each 3-cycle $(a,b,c)$ of a given resolution class of $C_i$,
$i=1,\ldots,s$, construct a 3-RGDD of type $t^{3}$ on $\{\{a\}\times
\{1,\ldots,t\}\}\cup \{\{b\}\times \{1,\ldots,t\}\}\cup
\{\{c\}\times \{1,\ldots,t\}\}$ having $t$ parallel classes of
3-cycles,  which comes from \cite{RS}. The result is a $(K_{1,3}, K_3)$-URGDD$(r,s)$
of $K_{(xt)(3)}$ into $rt$ parallel classes
of $3$-stars and $st$ parallel classes of 3-cycles.
\end{proof}

\begin{thm}
\label{thmFF} If there exists a $(K_{1,3}, K_3)$-URGDD$(r,s)$
of\/ $K_{x(3)}$, then for
each\/ $t\geq2$ there exists a $(K_{1,3}, K_3)$-URGDD$(r,s)$ of\/
$K_{(xt)(3)}$ into\/ $st$ parallel classes of\/ $3$-stars.

\end{thm}
\begin{proof}
The proof is similar to Theorem \ref{thmFFF}.
\end{proof}

\section{Small cases}

\begin{lemma}
\label{lemmaD1} There exists a\/ $(K_{1,3})$-URGDD$(4,0)$ of type\/ $2^{4}$.

\end{lemma}

\begin{proof}
Take the groups to be $\{0,1\}, \{2,3\}, \{4,5\}, \{6,7\}$ and the
classes as listed below:\\ $\{(0;2,4,6),(1;3,5,7)\}$, $\{
(2;4,1,6),(3;5,0,7)\}$, $\{(5;2,0,7),(4;1,3,6)\}$,\\
$\{(6;1,3,5),(7;0,4,2)\}$.
\end{proof}

\begin{lemma}
\label{lemmaD4} There exists a\/ $(K_{1,3}, K_3)$-URD$(12;4,2)$  of\/ $K_{12}-I$.
\end{lemma}

\begin{proof}
Let $V(K_{12}$)=$\mathbb{Z}_{12}$, $I$= $\{(1,10),(2,8),
(3,4),(5,11),
(6,9),(7,0)\}$  and the classes as listed below:\\
$\{(1,2,3)$, $(4,5,6)$, $(7,8,9)$, $(0,10,11)\}$,
$\{(1,5,9)$, $(4,8,11)$, $(2,6,0)$, $(3,7,10)\}$; \\
$\{(1;4,7,11)$, $(2;5,9,10)$, $(3;6,8,0)$\}$, $\{ $(4;2,9,10)$, $(5;3,8,0)$, $(6;1,7,11)\}$,\\
$\{(7;2,4,5)$, $(8;1,6,10)$, $(9;3,0,11)$\}$, $\{$(0;1,4,8)$,
$(10;5,6,9)$, $(11;2,3,7)\}$.
\end{proof}

\begin{lemma}
\label{lemmaD5} \label{lemmaD3} There exists a\/ $(K_{1,3}, K_3)$-URGDD$(8,2)$ of type\/
$8^{3}$.

\end{lemma}

\begin{proof}
Let $\{a_1,\ldots, a_8\}$, $\{b_1,\ldots, b_8\}$ and $\{c_1,\ldots,
c_8\}$ be the groups and the classes as listed below:

$\{(a_0,b_7,c_2)$, $(a_1,b_4,c_3)$, $(a_2,b_5,c_0)$, $(a_3,b_6,c_1)$,
$(a_4,b_3,c_6)$, $(a_5,b_0,c_7)$,\\$(a_6,b_1,c_4)$, $(a_7,b_2,c_5)\}$,

$\{(a_0,b_0,c_1)$, $(a_1,b_1,c_2)$, $(a_2,b_2,c_3)$, $(a_3,b_3,c_0)$,
$(a_4,b_4,c_5)$, $(a_5,b_5,c_6)$,\\$(a_6,b_6,c_7)$, $(a_7,b_7,c_4)\}$,

$\{(a_0;b_1,b_2,b_3)$, $(b_0;c_0,c_2,c_6)$, $(c_4;a_1,a_2,a_3)$,
$(b_7;c_1,c_5,c_7)$, $(c_3;a_4,a_5,a_6)$,\\$(a_7;b_4,b_5,b_6)\}$,

$\{(a_1;b_0,b_2,b_3)$, $(b_1;c_1,c_3,c_7)$, $(c_5;a_0,a_2,a_3)$,
$(b_4;c_2,c_4,c_6)$, $(c_0;a_7,a_5,a_6)$,\\ $(a_4;b_7,b_5,b_6)\}$,

$\{(a_2;b_1,b_0,b_3)$, $(b_2;c_0,c_2,c_4)$, $(c_6;a_1,a_0,a_3)$,
$(b_5;c_3,c_5,c_7)$, $(c_1;a_4,a_7,a_6)$,\\$(a_5;b_4,b_7,b_6)\}$,

$\{(a_3;b_1,b_2,b_0)$, $(b_3;c_1,c_3,c_5)$, $(c_7;a_1,a_2,a_0)$,
$(b_6;c_0,c_4,c_6)$, $(c_2;a_4,a_5,a_7)$, \\$(a_6;b_4,b_5,b_7) \}$,

$\{(a_0;b_4,b_5,c_4)$, $(b_7;c_0,c_7,a_1)$, $(c_2;b_6,a_2,a_3)$,
$(b_0;c_3,c_5,a_6)$, $(c_7;a_4,a_7,b_3)$, \\$(a_5;b_1,b_2,c_1)\}$,

$\{(a_1;b_5,b_6,c_5)$, $(b_4;c_1,a_2,c_7)$, $(c_3;a_3,a_0,b_7)$,
$(b_1;a_7,c_0,c_6)$, $(c_4;a_4,a_5,b_0)$, \\$(a_6;b_2,b_3,c_2) \}$,

$\{(a_2;b_6,b_7,c_6)$, $(b_5;a_3,c_2,c_4)$, $(c_0;a_1,a_0,b_4)$,
$(b_2;a_4,c_1,c_7)$, $(c_5;b_1,a_5,a_6)$,\\$(a_7;b_0,b_3,c_3)\}$,

$\{(a_3;b_4,b_7,c_7)$, $(b_6;a_0,c_3,c_5)$, $(c_1;a_1,a_2,b_5)$,
$(b_3;c_2,c_4,a_5)$, $(c_6;b_2,a_6,a_7)$,\\ $(a_4;c_0,b_0,b_1)\}$.

\end{proof}

\begin{lemma}
\label{lemmaD6} There exists a\/ $(K_{1,3}, K_3)$-URD$(24;12,2)$.

\end{lemma}

\begin{proof}
Take a $(K_{1,3}, K_3)$-URGDD$(8,2)$  of type $8^{3}$, which comes from Lemma
\ref{lemmaD5}. Place in each of the groups the same $(K_{1,3},
K_3)$-URGDD$(4,0)$ of type $2^{4}$, which
comes from Lemma \ref{lemmaD1}. This completes the proof.
\end{proof}

\begin{lemma}
\label{lemmaD8} There exists a\/ $(K_{1,3}, K_3)$-URD$(48;28,2)$.
\end{lemma}

\begin{proof}
Take a $K_{1,3}$-URGDD$(16,0)$ of type $24^2$
which is known to exist \cite{DW}. Place in each of the groups of
size 24 the same $(K_{1,3}, K_3)$-URD($24;12,2)$,  which comes from
Lemma \ref{lemmaD6}. This completes the proof.
\end{proof}

\begin{lemma}
\label{lemmaE1} There exists a $(K_{1,3}, K_3)$-URGDD$(16,0)$ of\/ $C_{3(12)}$.
\end{lemma}

\begin{proof}
Take the classes of 3-stars listed below:\\
$\{R^{1}_{0,0}\cup R^{2}_{3,0}\cup R^{3}_{3,1}\}$,
$\{R^{1}_{0,1}\cup R^{2}_{0,1}\cup R^{3}_{0,1}\}$,
$\{R^{1}_{0,2}\cup R^{2}_{1,2}\cup R^{3}_{1,1}\}$,\\
$\{R^{1}_{0,3}\cup R^{2}_{2,3}\cup R^{3}_{2,1}\}$,
$\{R^{1}_{1,0}\cup R^{2}_{3,1}\cup R^{3}_{0,2}\}$,
$\{R^{1}_{1,1}\cup R^{2}_{0,2}\cup R^{3}_{1,2}\}$,\\
$\{R^{1}_{1,2}\cup R^{2}_{1,3}\cup R^{3}_{2,2}\}$,
$\{R^{1}_{1,3}\cup R^{2}_{2,0}\cup R^{3}_{3,2}\}$,
$\{R^{1}_{2,0}\cup R^{2}_{3,2}\cup R^{3}_{1,3}\}$,\\
$\{R^{1}_{2,1}\cup R^{2}_{0,3}\cup R^{3}_{2,3}\}$,
$\{R^{1}_{2,2}\cup R^{2}_{1,0}\cup R^{3}_{3,3}\}$,
$\{R^{1}_{2,3}\cup R^{2}_{2,1}\cup R^{3}_{0,3}\}$,\\
$\{R^{1}_{3,0}\cup R^{2}_{3,3}\cup R^{3}_{2,0}\}$,
$\{R^{1}_{3,1}\cup R^{2}_{0,0}\cup R^{3}_{3,0}\}$,
$\{R^{1}_{3,2}\cup R^{2}_{1,1}\cup R^{3}_{0,0}\}$,\\
$\{R^{1}_{3,3}\cup R^{2}_{2,2}\cup R^{3}_{1,0}\}$.

\end{proof}

\begin{lemma}
\label{lemmaDD8}  There exists a\/ $(K_{1,3}, K_3)$-URD$(72;44,2)$.
\end{lemma}

\begin{proof}
Take a $K_{1,3}$-RGDD$(16)$  of type $12^3$ which comes from Lemma
\ref{lemmaE1}. Give weight $2$ to every point of this $K_{1,3}$-RGDD
and apply Theorem \ref{thmFF} with $t=2$. Place in each of the
groups of size 24 the same $(K_{1,3}, K_3)$-URD($24;12,2)$,  which
comes from Lemma \ref{lemmaD6}. Applying  Theorem \ref{thmDD} with
$g=12$, $t=2$ and $u=3$ we obtain the result.
\end{proof}

\begin{lemma}
\label{lemmaD9} There exists a\/ $(K_{1,3}, K_3)$-URD$(648;428,2)$.
\end{lemma}

\begin{proof}
Start with a $3$-RGDD of type $2^{27}$ \cite{RS}. Give weight $12$
to every point of this $3$-RGDD and place in each block of a given
resolution class of the $3$-RGDD the same $(K_{1,3})$-RGDD of type
$12^{3}$ with $16$ classes of $3$-stars, which comes from Lemma
\ref{lemmaE1}. Fill in each of the groups of sizes 24 with the same
$(K_{1,3}, K_3)$-URD($24;12,2)$. Applying Theorem \ref{thmDD} with
$g=2$, $t=12$ and $u=27$ we obtain a $(K_{1,3},
K_3)$-URD($648;428,2)$.
\end{proof}

\begin{lemma}
\label{lemmaD10} There exists a\/ $(K_{1,3}, K_3)$-URD$(816;540,2)$.
\end{lemma}

\begin{proof}
Start with a 4-RGDD of type $12^{34}$ \cite{CD}. Give weight $2$ to
every point of this $4$-RGDD and place in each block of a given
resolution class of the $4$-RGDD the same $K_{1,3}$-RGDD of type
$2^{4}$  with $4$ classes of $3$-stars, which comes from Lemma
\ref{lemmaD1}. Fill in each of the groups of sizes 24 with the same
$(K_{1,3}, K_3)$-URD($24;12,2)$. Applying Theorem \ref{thmDD} with
$g=12$, $t=2$ and $u=34$ we obtain a $(K_{1,3},
K_3)$-URD($816;540,2)$.
\end{proof}

\begin{lemma}
\label{lemmaE2} There exists a $(K_{1,3}, K_3)$-URGDD$(16,0)$ of\/
$C_{5(12)}$.
\end{lemma}

\begin{proof}
Take the classes of $3$-stars listed below:\\
$\{R^{1}_{0,0}\cup R^{2}_{3,0}\cup R^{3}_{3,0}\cup R^{4}_{3,0}\cup
R^{5}_{3,1}\}$, $\{R^{1}_{0,1}\cup R^{2}_{0,1}\cup
R^{3}_{0,1}\cup R^{4}_{0,1}\cup R^{5}_{0,1}\}$,\\
$\{R^{1}_{0,2}\cup R^{2}_{1,2}\cup R^{3}_{1,2}\cup R^{4}_{1,2}\cup
R^{5}_{1,1}\}$, $\{R^{1}_{0,3}\cup R^{2}_{2,3}\cup R^{3}_{2,3}\cup
R^{4}_{2,3}\cup R^{5}_{2,1}\}$,\\
$\{R^{1}_{1,0}\cup R^{2}_{3,1}\cup R^{3}_{0,2}\cup R^{4}_{1,3}\cup
R^{5}_{2,2}\}$, $\{R^{1}_{1,1}\cup R^{2}_{0,2}\cup
R^{3}_{1,3}\cup R^{4}_{2,0}\cup R^{5}_{3,2}\}$,\\
$\{R^{1}_{1,2}\cup R^{2}_{1,3}\cup R^{3}_{2,0}\cup R^{4}_{3,1}\cup
R^{5}_{0,2}\}$, $\{R^{1}_{1,3}\cup R^{2}_{2,0}\cup R^{3}_{3,1}\cup
R^{4}_{0,2}\cup R^{5}_{1,2}\}$,\\
$\{R^{1}_{2,0}\cup R^{2}_{3,2}\cup R^{3}_{1,0}\cup R^{4}_{3,2}\cup
R^{5}_{1,3}\}$, $\{R^{1}_{2,1}\cup R^{2}_{0,3}\cup
R^{3}_{2,1}\cup R^{4}_{0,3}\cup R^{5}_{2,3}\}$,\\
$\{R^{1}_{2,2}\cup R^{2}_{1,0}\cup R^{3}_{3,2}\cup R^{4}_{1,0}\cup
R^{5}_{3,3}\}$, $\{R^{1}_{2,3}\cup R^{2}_{2,1}\cup R^{3}_{0,3}\cup
R^{4}_{2,1}\cup R^{5}_{0,3}\}$,\\
$\{R^{1}_{3,0}\cup R^{2}_{3,3}\cup R^{3}_{2,2}\cup R^{4}_{1,1}\cup
R^{5}_{0,0}\}$, $\{R^{1}_{3,1}\cup R^{2}_{0,0}\cup
R^{3}_{3,3}\cup R^{4}_{2,2}\cup R^{5}_{1,0}\}$,\\
$\{R^{1}_{3,2}\cup R^{2}_{1,1}\cup R^{3}_{0,0}\cup R^{4}_{3,3}\cup
R^{5}_{2,0}\}$, $\{R^{1}_{3,3}\cup R^{2}_{2,2}\cup R^{3}_{1,1}\cup
R^{4}_{0,0}\cup R^{5}_{3,0}\}$.

\end{proof}

\begin{lemma}
\label{lemmaE3} There exists a $(K_{1,3}, K_3)$-URGDD$(16,0)$ of\/
$C_{7(12)}$.
\end{lemma}

\begin{proof}
Take the classes of 3-stars listed below:\\
$\{R^{1}_{0,0}\cup R^{2}_{3,0}\cup R^{3}_{3,0}\cup R^{4}_{3,0}\cup R^{5}_{3,0}\cup R^{6}_{3,0}\cup R^{7}_{3,1}\}$,\\
$\{R^{1}_{0,1}\cup R^{2}_{0,1}\cup R^{3}_{0,1}\cup R^{4}_{0,1}\cup R^{5}_{0,1}\cup R^{6}_{0,1}\cup R^{7}_{0,1}\}$,\\
$\{R^{1}_{0,2}\cup R^{2}_{1,2}\cup R^{3}_{1,2}\cup R^{4}_{1,2}\cup R^{5}_{1,2}\cup R^{6}_{1,2}\cup R^{7}_{1,1}\}$,\\
$\{R^{1}_{0,3}\cup R^{2}_{2,3}\cup R^{3}_{2,3}\cup R^{4}_{2,3}\cup R^{5}_{2,3}\cup R^{6}_{2,3}\cup R^{7}_{2,1}\}$,\\
$\{R^{1}_{1,0}\cup R^{2}_{3,1}\cup R^{3}_{0,2}\cup R^{4}_{1,3}\cup R^{5}_{2,0}\cup R^{6}_{3,1}\cup R^{7}_{0,2}\}$,\\
$\{R^{1}_{1,1}\cup R^{2}_{0,2}\cup R^{3}_{1,3}\cup R^{4}_{2,0}\cup R^{5}_{3,1}\cup R^{6}_{0,2}\cup R^{7}_{1,2}\}$,\\
$\{R^{1}_{1,2}\cup R^{2}_{1,3}\cup R^{3}_{2,0}\cup R^{4}_{3,1}\cup R^{5}_{0,2}\cup R^{6}_{1,3}\cup R^{7}_{2,2}\}$,\\
$\{R^{1}_{1,3}\cup R^{2}_{2,0}\cup R^{3}_{3,1}\cup R^{4}_{0,2}\cup R^{5}_{1,3}\cup R^{6}_{2,0}\cup R^{7}_{3,2}\}$,\\
$\{R^{1}_{2,0}\cup R^{2}_{3,2}\cup R^{3}_{1,0}\cup R^{4}_{3,2}\cup R^{5}_{1,0}\cup R^{6}_{3,2}\cup R^{7}_{1,3}\}$,\\
$\{R^{1}_{2,1}\cup R^{2}_{0,3}\cup R^{3}_{2,1}\cup R^{4}_{0,3}\cup R^{5}_{2,1}\cup R^{6}_{0,3}\cup R^{7}_{2,3}\}$,\\
$\{R^{1}_{2,2}\cup R^{2}_{1,0}\cup R^{3}_{3,2}\cup R^{4}_{1,0}\cup R^{5}_{3,2}\cup R^{6}_{1,0}\cup R^{7}_{3,3}\}$,\\
$\{R^{1}_{2,3}\cup R^{2}_{2,1}\cup R^{3}_{0,3}\cup R^{4}_{2,1}\cup R^{5}_{0,3}\cup R^{6}_{2,1}\cup R^{7}_{0,3}\}$,\\
$\{R^{1}_{3,0}\cup R^{2}_{3,3}\cup R^{3}_{2,2}\cup R^{4}_{1,1}\cup R^{5}_{0,0}\cup R^{6}_{3,3}\cup R^{7}_{2,0}\}$,\\
$\{R^{1}_{3,1}\cup R^{2}_{0,0}\cup R^{3}_{3,3}\cup R^{4}_{2,2}\cup R^{5}_{1,1}\cup R^{6}_{0,0}\cup R^{7}_{3,0}\}$,\\
$\{R^{1}_{3,2}\cup R^{2}_{1,1}\cup R^{3}_{0,0}\cup R^{4}_{3,3}\cup R^{5}_{2,2}\cup R^{6}_{1,1}\cup R^{7}_{0,0}\}$,\\
$\{R^{1}_{3,3}\cup R^{2}_{2,2}\cup R^{3}_{1,1}\cup R^{4}_{0,0}\cup
R^{5}_{3,3}\cup R^{6}_{2,2}\cup R^{7}_{1,0}\}$.

\end{proof}

\begin{lemma}
\label{lemmaE4} There exists a $(K_{1,3}, K_3)$-URGDD$(16,0)$ of\/
$C_{9(12)}$.
\end{lemma}

\begin{proof}
Take the classes of 3-stars listed below:\\
$\{R^{1}_{0,0}\cup R^{2}_{3,0}\cup R^{3}_{3,0}\cup R^{4}_{3,0}\cup
R^{5}_{3,0}\cup R^{6}_{3,0}\cup R^{7}_{3,0}\cup R^{8}_{3,0}\cup
R^{9}_{3,1}\}$,
\\$\{R^{1}_{0,1}\cup R^{2}_{0,1}\cup R^{3}_{0,1}\cup
R^{4}_{0,1}\cup R^{5}_{0,1}\cup R^{6}_{0,1}\cup R^{7}_{0,1}\cup
R^{8}_{0,1}\cup R^{9}_{0,1}\}$,
\\$\{R^{1}_{0,2}\cup R^{2}_{1,2}\cup R^{3}_{1,2}\cup
R^{4}_{1,2} \cup R^{5}_{1,2}\cup R^{6}_{1,2}\cup R^{7}_{1,2}\cup
R^{8}_{1,2}\cup R^{9}_{1,1}\}$,
\\$\{R^{1}_{0,3}\cup R^{2}_{2,3}\cup R^{3}_{2,3}\cup
R^{4}_{2,3}\cup R^{5}_{2,3}\cup R^{6}_{2,3}\cup R^{7}_{2,3}\cup
R^{8}_{2,3}\cup R^{9}_{2,1}\}$,
\\$\{R^{1}_{1,0}\cup R^{2}_{3,1}\cup R^{3}_{0,2}\cup
R^{4}_{1,3}\cup R^{5}_{2,0}\cup R^{6}_{3,1}\cup R^{7}_{0,2}\cup
R^{8}_{1,3}\cup R^{9}_{2,2}\}$,
\\$\{R^{1}_{1,1}\cup R^{2}_{0,2}\cup
R^{3}_{1,3}\cup R^{4}_{2,0}\cup R^{5}_{3,1}\cup R^{6}_{0,2}\cup
R^{7}_{1,3}\cup R^{8}_{2,0}\cup R^{9}_{3,2}\}$,
\\$\{R^{1}_{1,2}\cup R^{2}_{1,3}\cup R^{3}_{2,0}\cup
R^{4}_{3,1}\cup R^{5}_{0,2}\cup R^{6}_{1,3}\cup R^{7}_{2,0}\cup
R^{8}_{3,1}\cup R^{9}_{0,2}\}$,
\\$\{R^{1}_{1,3}\cup R^{2}_{2,0}\cup R^{3}_{3,1}\cup R^{4}_{0,2}\cup
R^{5}_{1,3}\cup R^{6}_{2,0}\cup R^{7}_{3,1}\cup R^{8}_{0,2}\cup
R^{9}_{1,2}\}$,
\\$\{R^{1}_{2,0}\cup R^{2}_{3,2}\cup R^{3}_{1,0}\cup
R^{4}_{3,2}\cup R^{5}_{1,0}\cup R^{6}_{3,2}\cup R^{7}_{1,0}\cup
R^{8}_{3,2}\cup R^{9}_{1,3}\}$,
\\$\{R^{1}_{2,1}\cup R^{2}_{0,3}\cup
R^{3}_{2,1}\cup R^{4}_{0,3}\cup R^{5}_{2,1}\cup R^{6}_{0,3}\cup
R^{7}_{2,1}\cup R^{8}_{0,3}\cup R^{9}_{2,3}\}$,
\\$\{R^{1}_{2,2}\cup R^{2}_{1,0}\cup R^{3}_{3,2}\cup
R^{4}_{1,0}\cup R^{5}_{3,2}\cup R^{6}_{1,0}\cup R^{7}_{3,2}\cup
R^{8}_{1,0}\cup R^{9}_{3,3}\}$,
\\$\{R^{1}_{2,3}\cup R^{2}_{2,1}\cup R^{3}_{0,3}\cup
R^{4}_{2,1}\cup R^{5}_{0,3}\cup R^{6}_{2,1}\cup R^{7}_{0,3}\cup
R^{8}_{2,1}\cup R^{9}_{0,3}\}$,
\\$\{R^{1}_{3,0}\cup R^{2}_{3,3}\cup R^{3}_{2,2}\cup
R^{4}_{1,1}\cup R^{5}_{0,0}\cup R^{6}_{3,3}\cup R^{7}_{2,2}\cup
R^{8}_{1,1}\cup R^{9}_{0,0}\}$,
\\ $\{R^{1}_{3,1}\cup R^{2}_{0,0}\cup R^{3}_{3,3}\cup R^{4}_{2,2}\cup R^{5}_{1,1}\cup
R^{6}_{0,0}\cup R^{7}_{3,3}\cup R^{8}_{2,2}\cup R^{9}_{1,0}\}$,
\\$\{R^{1}_{3,2}\cup R^{2}_{1,1}\cup R^{3}_{0,0}\cup
R^{4}_{3,3}\cup R^{5}_{2,2}\cup R^{6}_{1,1}\cup R^{7}_{0,0}\cup
R^{8}_{3,3}\cup R^{9}_{2,0}\}$,
\\$\{R^{1}_{3,3}\cup R^{2}_{2,2}\cup R^{3}_{1,1}\cup
R^{4}_{0,0}\cup R^{5}_{3,3}\cup R^{6}_{2,2}\cup R^{7}_{1,1}\cup
R^{8}_{0,0}\cup R^{9}_{3,0}\}$.

\end{proof}

\begin{lemma}
\label{lemmaE5} There exists a $(K_{1,3}, K_3)$-URGDD$(16,0)$ of\/
$C_{11(12)}$.
\end{lemma}

\begin{proof}
Take the classes of 3-stars listed below:
\\ $\{R^{1}_{0,0}\cup R^{2}_{3,0}\cup R^{3}_{3,0}\cup R^{4}_{3,0}\cup R^{5}_{3,0}\cup R^{6}_{3,0}\cup R^{7}_{3,0}\cup R^{8}_{3,0}
\cup R^{9}_{3,0}\cup R^{10}_{3,0}\cup R^{11}_{3,1}\}$,
\\ $\{R^{1}_{0,1}\cup R^{2}_{0,1}\cup R^{3}_{0,1}\cup R^{4}_{0,1}\cup R^{5}_{0,1}\cup R^{6}_{0,1}\cup R^{7}_{0,1}\cup R^{8}_{0,1}
\cup R^{9}_{0,1}\cup R^{10}_{0,1}\cup R^{11}_{0,1}\}$,
\\ $\{R^{1}_{0,2}\cup R^{2}_{1,2}\cup R^{3}_{1,2}\cup R^{4}_{1,2}\cup R^{5}_{1,2}\cup R^{6}_{1,2}\cup R^{7}_{1,2}\cup R^{8}_{1,2}
\cup R^{9}_{1,2}\cup R^{10}_{1,2}\cup R^{11}_{1,1}\}$,
\\ $\{R^{1}_{0,3}\cup R^{2}_{2,3}\cup R^{3}_{2,3}\cup R^{4}_{2,3}\cup R^{5}_{2,3}\cup R^{6}_{2,3}\cup R^{7}_{2,3}\cup R^{8}_{2,3}
\cup R^{9}_{2,3}\cup R^{10}_{2,3}\cup R^{11}_{2,1}\}$,\\
$\{R^{1}_{1,0}\cup R^{2}_{3,1}\cup R^{3}_{0,2}\cup R^{4}_{1,3}\cup
R^{5}_{2,0}\cup R^{6}_{3,1}\cup R^{7}_{0,2}\cup R^{8}_{1,3}
\cup R^{9}_{2,0}\cup R^{10}_{3,1}\cup R^{11}_{0,2}\}$,\\
$\{R^{1}_{1,1}\cup R^{2}_{0,2}\cup R^{3}_{1,3}\cup R^{4}_{2,0}\cup
R^{5}_{3,1}\cup R^{6}_{0,2}\cup R^{7}_{1,3}\cup R^{8}_{2,0}
\cup R^{9}_{3,1}\cup R^{10}_{0,2}\cup R^{11}_{1,2}\}$,\\
$\{R^{1}_{1,2}\cup R^{2}_{1,3}\cup R^{3}_{2,0}\cup R^{4}_{3,1}\cup
R^{5}_{0,2}\cup R^{6}_{1,3}\cup R^{7}_{2,0}\cup R^{8}_{3,1}
\cup R^{9}_{0,2}\cup R^{10}_{1,3}\cup R^{11}_{2,2}\}$,\\
$\{R^{1}_{1,3}\cup R^{2}_{2,0}\cup R^{3}_{3,1}\cup R^{4}_{0,2}\cup
R^{5}_{1,3}\cup R^{6}_{2,0}\cup R^{7}_{3,1}\cup R^{8}_{0,2}
\cup R^{9}_{1,3}\cup R^{10}_{2,0}\cup R^{11}_{3,2}\}$,\\
$\{R^{1}_{2,0}\cup R^{2}_{3,2}\cup R^{3}_{1,0}\cup R^{4}_{3,2}\cup
R^{5}_{1,0}\cup R^{6}_{3,2} \cup R^{7}_{1,0}\cup R^{8}_{3,2}
\cup R^{9}_{1,0}\cup R^{10}_{3,2}\cup R^{11}_{1,3}\}$,\\
$\{R^{1}_{2,1}\cup R^{2}_{0,3}\cup R^{3}_{2,1}\cup R^{4}_{0,3}\cup
R^{5}_{2,1}\cup R^{6}_{0,3}\cup R^{7}_{2,1}\cup R^{8}_{0,3}
\cup R^{9}_{2,1}\cup R^{10}_{0,3} \cup R^{11}_{2,3}\}$,\\
$\{R^{1}_{2,2}\cup R^{2}_{1,0}\cup R^{3}_{3,2}\cup R^{4}_{1,0}\cup
R^{5}_{3,2}\cup R^{6}_{1,0}\cup R^{7}_{3,2}\cup R^{8}_{1,0}
\cup R^{9}_{3,2}\cup R^{10}_{1,0} \cup R^{11}_{3,3}\}$,\\
$\{R^{1}_{2,3}\cup R^{2}_{2,1}\cup R^{3}_{0,3}\cup R^{4}_{2,1}\cup
R^{5}_{0,3}\cup R^{6}_{2,1}\cup R^{7}_{0,3}\cup R^{8}_{2,1}
\cup R^{9}_{0,3}\cup R^{10}_{2,1}\cup R^{11}_{0,3}\}$,\\
$\{R^{1}_{3,0}\cup R^{2}_{3,3}\cup R^{3}_{2,2}\cup R^{4}_{1,1}\cup
R^{5}_{0,0}\cup R^{6}_{3,3}\cup R^{7}_{2,2}\cup R^{8}_{1,1}
\cup R^{9}_{0,0}\cup R^{10}_{3,3} \cup R^{11}_{2,0}\}$,\\
$\{R^{1}_{3,1}\cup R^{2}_{0,0}\cup R^{3}_{3,3}\cup R^{4}_{2,2}\cup
R^{5}_{1,1}\cup R^{6}_{0,0}\cup R^{7}_{3,3}\cup R^{8}_{2,2}
\cup R^{9}_{1,1}\cup R^{10}_{0,0}\cup R^{11}_{3,0}\}$,\\
$\{R^{1}_{3,2}\cup R^{2}_{1,1}\cup R^{3}_{0,0}\cup R^{4}_{3,3}\cup
R^{5}_{2,2}\cup R^{6}_{1,1}\cup R^{7}_{0,0}\cup R^{8}_{3,3}
\cup R^{9}_{2,2}\cup R^{10}_{1,1}\cup R^{11}_{0,0}\}$,\\
$\{R^{1}_{3,3}\cup R^{2}_{2,2}\cup R^{3}_{1,1}\cup R^{4}_{0,0}\cup
R^{5}_{3,3}\cup R^{6}_{2,2}\cup R^{7}_{1,1}\cup R^{8}_{0,0} \cup
R^{9}_{3,3}\cup R^{10}_{2,2} \cup R^{11}_{1,0}\}$.

\end{proof}

\section{Main results}

\begin{lemma}
\label{lemmaC1} For every\/ $v\equiv 0\pmod{24}$ there exists a\/
$(K_{1,3}, K_3)$-URD $(v;\frac{2(v-6)}{3},2)$.

\end{lemma}

\begin{proof}
Let $v=24s$. The cases $s=1,2,3,27,34$ are covered by Lemmas
\ref{lemmaD6}, \ref{lemmaD8}, \ref{lemmaDD8}, \ref{lemmaD9} and
\ref{lemmaD10}. For $s>4$, $s\neq 27,\,34$, start with a $4$-RGDD
$G$ of type $6^{2s}$ \cite{CD}. Give weight 2 to each point of this
4-RGDD and place in each block of a given resolution class of the
4-RGDD the same $(K_{1,3}, K_3)$-URGDD$(4,0)$ of type $2^{4}$, which
comes from Lemma \ref{lemmaD1}. Fill in each of the groups of sizes
12 with the same $(K_{1,3},K_3)$-URD($12;4,2)$, which comes from
Lemma \ref{lemmaD4}. Applying Theorem \ref{thmDD} with $g=6$, $t=2$
and $u=2s$ we obtain a uniformly resolvable decomposition of
$K_{v}-I$, $I=\bigcup_{i=1}^{2s} I_{i}$, into
$8(2s-1)+4=\frac{2(v-6)}{3}$ classes of $3$-stars and $2$ classes of
$3$-cycles.
\end{proof}

\begin{lemma}
\label{lemmaE6} There exists a $(K_{1,3}, K_3)$-URGDD$(16,0)$
 of\/
$K_{(3+4p)(12)}$,\/ $p > 2$.
\end{lemma}

\begin{proof}
Take the classes of 3-stars listed below:
\\$\{R^{1}_{0,0}\cup\{\bigcup_{j=2}^{2+4p} R^{j}_{3,0}\}\cup R^{3+4p}_{3,1}\}$, $\{\bigcup_{j=1}^{3+4p} R^{j}_{0,1}\}$,
\\ $\{R^{1}_{0,2}\cup\{\bigcup_{j=2}^{2+4p} R^{j}_{1,2}\}\cup R^{3+4p}_{1,1}\}$, $\{R^{1}_{0,3}\cup\{\bigcup_{j=2}^{2+4p} R^{j}_{2,3}\}\cup R^{3+4p}_{2,1}\}$,
\\$\{R^{1}_{1,0}\cup\{\bigcup_{j=0}^p
R^{2+4j}_{3,1}\}\cup\{\bigcup_{j=0}^{p-1}
R^{3+4j}_{0,2}\}\cup\{\bigcup_{j=0}^{p-1} R^{4+4j}_{1,3}\}
\cup\{\bigcup_{j=0}^{p-1} R^{5+4j}_{2,0}\} \cup R^{3+4p}_{0,2}\}$,
\\$\{R^{1}_{1,1}\cup\{\bigcup_{j=0}^p
R^{2+4j}_{0,2}\}\cup\{\bigcup_{j=0}^{p-1}
R^{3+4j}_{1,3}\}\cup\{\bigcup_{j=0}^{p-1} R^{4+4j}_{2,0}\}
\cup\{\bigcup_{j=0}^{p-1} R^{5+4j}_{3,1}\}\cup R^{3+4p}_{1,2}\}$,
\\$\{R^{1}_{1,2}\cup\{\bigcup_{j=0}^p R^{2+4j}_{1,3}\}\cup(\bigcup_{j=0}^{p-1}
R^{3+4j}_{2,0}\}\cup\{\bigcup_{j=0}^{p-1} R^{4+4j}_{3,1}\}
\cup\{\bigcup_{j=0}^{p-1} R^{5+4j}_{0,2}\} \cup R^{3+4p}_{2,2}\}$,\\
$\{R^{1}_{1,3}\cup\{\bigcup_{j=0}^p
R^{2+4j}_{2,0}\}\cup\{\bigcup_{j=0}^{p-1}
R^{3+4j}_{3,1}\}\cup\{\bigcup_{j=0}^{p-1} R^{4+4j}_{0,2}\}
\cup\{\bigcup_{j=0}^{p-1} R^{5+4j}_{1,3}\} \cup R^{3+4p}_{3,2}\}$,
\\ $\{R^{1}_{2,0}\cup\{\bigcup_{j=0}^{2p} R^{2+2j}_{3,2}\}\cup\{\bigcup_{j=0}^{2p-1} R^{3+2j}_{1,0}\}\cup R^{3+4p}_{1,3}\}$,\\
$\{R^{1}_{2,1}\cup\{\bigcup_{j=0}^{2p} R^{2+2j}_{0,3}\}\cup\{\bigcup_{j=0}^{2p-1} R^{3+2j}_{2,1}\}\cup R^{3+4p}_{2,3}\}$,\\
$\{R^{1}_{2,2}\cup\{\bigcup_{j=0}^{2p} R^{2+2j}_{1,0}\}\cup\{\bigcup_{j=0}^{2p-1} R^{3+2j}_{3,2}\}\cup R^{3+4p}_{3,3}\}$,\\
$\{R^{1}_{2,3}\cup\{\bigcup_{j=0}^{2p} R^{2+2j}_{2,1}\}\cup\{\bigcup_{j=0}^{2p-1} R^{3+2j}_{0,3}\}\cup R^{3+4p}_{0,3}\}$,\\
$\{R^{1}_{3,0}\cup\{\bigcup_{j=0}^p
R^{2+4j}_{3,3}\}\cup\{\bigcup_{j=0}^{p-1}
R^{3+4j}_{2,2}\}\cup\{\bigcup_{j=0}^{p-1} R^{4+4j}_{1,1}\}
\cup\{\bigcup_{j=0}^{p-1} R^{5+4j}_{0,0}\} \cup R^{3+4p}_{2,0}\}$,\\
$\{R^{1}_{3,1}\cup\{\bigcup_{j=0}^p
R^{2+4j}_{0,0}\}\cup\{\bigcup_{j=0}^{p-1}
R^{3+4j}_{3,3}\}\cup\{\bigcup_{j=0}^{p-1} R^{4+4j}_{2,2}\}
\cup\{\bigcup_{j=0}^{p-1} R^{5+4j}_{1,1}\} \cup R^{3+4p}_{3,0}\}$,\\
$\{R^{1}_{3,2}\cup\{\bigcup_{j=0}^p
R^{2+4j}_{1,1}\}\cup\{\bigcup_{j=0}^{p-1}
R^{3+4j}_{0,0}\}\cup\{\bigcup_{j=0}^{p-1} R^{4+4j}_{3,3}\}
\cup\{\bigcup_{j=0}^{p-1} R^{5+4j}_{2,2}\}\cup R^{3+4p}_{0,0}\}$,\\
$\{R^{1}_{3,3}\cup\{\bigcup_{j=0}^p
R^{2+4j}_{2,2}\}\cup\{\bigcup_{j=0}^{p-1}
R^{3+4j}_{1,1}\}\cup\{\bigcup_{j=0}^{p-1} R^{4+4j}_{0,0}\}
\cup\{\bigcup_{j=0}^{p-1} R^{5+4j}_{3,3}\} \cup R^{3+4p}_{1,0}\}$.

\end{proof}

\begin{lemma}
\label{lemmaE7} There exists a $(K_{1,3}, K_3)$-URGDD$(16,0)$ of\/
$C_{(1+4p)(12)}$,\/ $p > 2$.
\end{lemma}

\begin{proof}
Take the classes of 3-stars listed below:
\\$\{R^{1}_{0,0}\cup\{\bigcup_{j=2}^{4p} R^{j}_{3,0}\}\cup R^{1+4p}_{3,1}\}$, $\{\bigcup_{j=1}^{1+4p} R^{j}_{0,1}\}$,
\\ $\{R^{1}_{0,2}\cup\{\bigcup_{j=2}^{4p} R^{j}_{1,2}\}\cup R^{1+4p}_{1,1}\}$, $\{R^{1}_{0,3}\cup\{\bigcup_{j=2}^{4p} R^{j}_{2,3}\}\cup R^{1+4p}_{2,1}\}$,\\
$\{R^{1}_{1,0}\cup\{\bigcup_{j=0}^{p-1}
R^{2+4j}_{3,1}\}\cup\{\bigcup_{j=0}^{p-1}
R^{3+4j}_{0,2}\}\cup\{\bigcup_{j=0}^{p-1} R^{4+4j}_{1,3}\}
\cup\{\bigcup_{j=0}^{p-2} R^{5+4j}_{2,0}\} \cup R^{1+4p}_{2,2}\}$,
\\$\{R^{1}_{1,1}\cup\{\bigcup_{j=0}^{p-1}
R^{2+4j}_{0,2}\}\cup\{\bigcup_{j=0}^{p-1}
R^{3+4j}_{1,3}\}\cup\{\bigcup_{j=0}^{p-1} R^{4+4j}_{2,0}\}
\cup\{\bigcup_{j=0}^{p-2} R^{5+4j}_{3,1}\}\cup R^{1+4p}_{3,2}\}$,
\\$\{R^{1}_{1,2}\cup\{\bigcup_{j=0}^{p-1}R^{2+4j}_{1,3}\}\cup(\bigcup_{j=0}^{p-1}
R^{3+4j}_{2,0}\}\cup\{\bigcup_{j=0}^{p-1} R^{4+4j}_{3,1}\}
\cup\{\bigcup_{j=0}^{p-2} R^{5+4j}_{0,2}\} \cup R^{1+4p}_{0,2}\}$,\\
$\{R^{1}_{1,3}\cup\{\bigcup_{j=0}^{p-1}
R^{2+4j}_{2,0}\}\cup\{\bigcup_{j=0}^{p-1}
R^{3+4j}_{3,1}\}\cup\{\bigcup_{j=0}^{p-1} R^{4+4j}_{0,2}\}
\cup\{\bigcup_{j=0}^{p-2} R^{5+4j}_{1,3}\} \cup R^{1+4p}_{1,2}\}$,\\
$\{R^{1}_{2,0}\cup\{\bigcup_{j=0}^{2p-1} R^{2j+2}_{3,2}\}\cup\{\bigcup_{j=0}^{2p-2} R^{3+2j}_{1,0}\}\cup R^{1+4p}_{1,3}\}$,\\
$\{R^{1}_{2,1}\cup\{\bigcup_{j=0}^{2p-1} R^{2+2j}_{0,3}\}\cup\{\bigcup_{j=0}^{2p-2} R^{3+2j}_{2,1}\}\cup R^{1+4p}_{2,3}\}$,\\
$\{R^{1}_{2,2}\cup\{\bigcup_{j=0}^{2p-1} R^{2+2j}_{1,0}\}\cup\{\bigcup_{j=0}^{2p-2} R^{3+2j}_{3,2}\}\cup R^{1+4p}_{3,3}\}$,\\
$\{R^{1}_{2,3}\cup\{\bigcup_{j=0}^{2p-1} R^{2+2j}_{2,1}\}\cup\{\bigcup_{j=0}^{2p-2} R^{3+2j}_{0,3}\}\cup R^{1+4p}_{0,3}\}$,\\
$\{R^{1}_{3,0}\cup\{\bigcup_{j=0}^{p-1}
R^{2+4j}_{3,3}\}\cup\{\bigcup_{j=0}^{p-1}
R^{3+4j}_{2,2}\}\cup\{\bigcup_{j=0}^{p-1} R^{4+4j}_{1,1}\}
\cup\{\bigcup_{j=0}^{p-2} R^{5+4j}_{0,0}\} \cup R^{1+4p}_{0,0}\}$,\\
$\{R^{1}_{3,1}\cup\{\bigcup_{j=0}^{p-1}
R^{2+4j}_{0,0}\}\cup\{\bigcup_{j=0}^{p-1}
R^{3+4j}_{3,3}\}\cup\{\bigcup_{j=0}^{p-1} R^{4+4j}_{2,2}\}
\cup\{\bigcup_{j=0}^{p-2} R^{5+4j}_{1,1}\} \cup R^{1+4p}_{1,0}\}$,\\
$\{R^{1}_{3,2}\cup\{\bigcup_{j=0}^{p-1}
R^{2+4j}_{1,1}\}\cup\{\bigcup_{j=0}^{p-1}
R^{3+4j}_{0,0}\}\cup\{\bigcup_{j=0}^{p-1} R^{4+4j}_{3,3}\}
\cup\{\bigcup_{j=0}^{p-2} R^{5+4j}_{2,2}\}\cup R^{1+4p}_{2,0}\}$,\\
$\{R^{1}_{3,3}\cup\{\bigcup_{j=0}^{p-1}
R^{2+4j}_{2,2}\}\cup\{\bigcup_{j=0}^{p-1}
R^{3+4j}_{1,1}\}\cup\{\bigcup_{j=0}^{p-1} R^{4+4j}_{0,0}\}
\cup\{\bigcup_{j=0}^{p-2} R^{5+4j}_{3,3}\} \cup R^{1+4p}_{3,0}\}$.

\end{proof}

\begin{lemma}
\label{lemmaC2} For every\/ $v\equiv 12\pmod{24}$ there exists a\/
$(K_{1,3}, K_3)$-URD $(v;\frac{2(v-6)}{3},2)$.

\end{lemma}

\begin{proof}
Let $v=12(2t+1)$, $t\geq0$. The case $v=12$ is covered by Lemma~\ref{lemmaD4}.
 For $t\geq 1$ start with a $(2t+1)$-cycle system
$(X,C)$ \cite{A,L}. Give weight 12 to each point of $X$ and replace
each $(2t+1)$-cycle of $C$ with a $(K_{1,3}, K_3)$-URGDD$(16,0)$
of $C_{(1+2t)12}$, which comes from
Lemmas \ref{lemmaE1}, \ref{lemmaE2}, \ref{lemmaE3}, \ref{lemmaE4},
\ref{lemmaE5}, \ref{lemmaE6} and \ref{lemmaE7}. For each $a_i\in X$,
$i=0,\ldots,2t$, place in $a_i\times\Z_{12}$ the same URD$(12;4,2)$,
which comes from Lemma \ref{lemmaD4}. Since $|C|=t$, the result is a
uniformly resolvable decomposition of $K_{12+24t}-I$,
$I=\bigcup_{i=0}^{2t}I_{i}$, into $16t+4=\frac{2(v-6)}{3}$ classes of
$3$-stars and $2$ classes of $3$-cycles.
\end{proof}

Combining Lemmas \ref{lemmaC1} and \ref{lemmaC2}  we obtain the main
theorem of this article.

\begin{thm}
For every\/ $v\equiv 0\pmod{12}$, there exists a\/ $(K_{1,3},
K_3)$-URD $(v;\frac{2(v-6)}{3},2)$.
\end{thm}

\end{document}